\documentclass{elsarticle}
\usepackage{amssymb, latexsym, comment, amsmath, amsthm, upgreek, epsf, epsfig, color, natbib, xfrac}

\makeatletter
\def\ps@pprintTitle{%
\let\@oddhead\@empty
\let\@evenhead\@empty
\def\@oddfoot{\centerline{\thepage}}%
\let\@evenfoot\@oddfoot}
\makeatother

\newtheorem{thm}{Theorem}[section]
\newtheorem{lemma}[thm]{Lemma}
\newtheorem{prop}[thm]{Proposition}

\newtheorem{definition}[thm]{Definition}
\newtheorem{remark}[thm]{Remark}
\newtheorem{ex}[thm]{Example}

\newcommand{\ma}{\measuredangle}

\begin{document}
\begin{frontmatter}
\title{Morse Theory for the Uniform Energy}
\author{Ian M.~Adelstein \\ Jonathan Epstein} 
\address{Department of Mathematics, Yale University\\ New Haven, CT 06520 United States}
\address{Department of Mathematics, University of Oklahoma\\ Norman, OK 73019 United States}
\begin{abstract} In this paper we develop a Morse theory for the uniform energy. We use the one-sided directional derivative of the distance function to study the minimizing properties of variations through closed geodesics. This derivative is then used to define a one-sided directional derivative for the uniform energy which allows us to identify gradient-like vectors at those points where the function is not differentiable. These vectors are used to restart the standard negative gradient flow of the Morse energy at its critical points. We illustrate this procedure on the flat torus and demonstrate that the restarted flow improves the minimizing properties of the associated closed geodesics.



\end{abstract}
\end{frontmatter}

\section{Introduction}

The existence of closed geodesics on Riemannian manifolds is one of the foundational questions in the study of global differential geometry. The results in this field arise from the study of the critical points of the Morse energy function defined on the loop space $\Omega$ of a Riemannian manifold $M$. 

\begin{definition} The \emph{Morse energy function} $E \colon \Omega \to \mathbb{R}$ is defined for a piecewise smooth closed curve $c \colon S^1 \to M$ as $$E(c)= \int_{S^1} \lvert \dot{c}(t) \rvert^2 \, dt$$
\end{definition}

A standard application of the first variation formula shows that the critical points of this function are closed geodesics (cf.~\cite{Mil}, Corollary 12.3). The first result in this field came at the end of the 19\textsuperscript{th} century and is attributed alternately to Cartan (cf.~\cite{doC}, Theorem 12.2.2) and Hadamard (cf.~\cite{Ber}, Theorem 202). They showed that the shortest curve in each nontrivial homotopy class is a closed geodesic. The modern proof of this result is an application of the negative gradient flow of the Morse energy, a procedure that deforms a closed curve within its homotopy class in the direction of maximal decrease of the Morse energy. We develop an extension to this negative gradient flow of the Morse energy that allows us to restart the flow at its critical points. We show by example that the restarted flow works to improve the minimizing properties of the associated closed geodesics. The negative gradient flow appears in the literature under various important forms, including as the curve shortening flow and the mean curvature flow (cf.~\cite{Cold}).



A defining property of a geodesic is that it is a locally distance minimizing curve. It is clear that a nontrivial closed geodesic can never be a \emph{globally} distance minimizing curve. Indeed, a closed geodesic can never minimize past half its length: traversing the geodesic in the opposite direction will always provide a shorter path. It is therefore natural to consider the largest interval on which a given closed geodesic is distance minimizing. This led Sormani \cite{Sor} to consider the notion of a 1/k-geodesic.

\begin{definition} A \emph{1/k-geodesic} is a closed geodesic $\gamma \colon S^1 \to M$ which is minimizing on all subintervals of length $l(\gamma)/k$, i.e. $$d(\gamma(t),\gamma(t+2\pi/k))=l(\gamma)/k \hspace{4mm} \forall t \in S^1 = \mathbb{R}/2\pi \mathbb{Z} .$$ 
\end{definition}



\begin{definition}\label{openly} An \emph{openly} 1/k-geodesic is a 1/k-geodesic that does not contain cut points at distance $l(\gamma)/k$, and therefore will always minimize on some open neighborhood of subintervals of length $l(\gamma)/k$. 
\end{definition}

As a first example we note that the great circles on the sphere are 1/2-geodesics but are not openly 1/2-geodesics. Additionally, by the compactness of $S^1$ and the local minimizing property of geodesics, we have that every closed geodesic is a 1/k-geodesic for some $k \geq 2$. Sormani \cite[Theorem 10.2]{Sor} demonstrated a one-to-one relationship between the openly 1/k-geodesics and the rotating smooth critical points of the uniform energy $E^k \colon M^k \to \mathbb{R}$, a finite dimensional approximation to the Morse energy function (cf.~\cite{Mil}, Theorem 16.2). 

\begin{definition}\label{ue} The \emph{uniform energy} $E^k \colon M^k \to \mathbb{R}$ is defined for an element $\bar{x} = (x_1, \ldots , x_k)  \in M^k $ as 
\begin{equation*}E^k(\bar{x}) = \sum_{i=1}^k \frac{d(x_i , x_{i+1})^2}{1/k} \hspace{9pt} where \hspace{6pt} x_{k+1} = x_1
\end{equation*}
\end{definition}

Sormani then asked \cite[Remark 10.5]{Sor} if there was a relationship between the degenerate smooth critical points of the uniform energy (Definition~\ref{degenerate}) and the minimizing properties of those curves which are close in a variational sense to an openly 1/k-geodesic. We demonstrate such a relationship in Remark~\ref{rem1} by first proving the following.




\begin{thm}\label{open_var} Let $\gamma \colon S^1 \to M$ be an openly 1/k-geodesic that is degenerate in the sense that there exists a variation $\alpha(s,t)$ of $\gamma$ through closed geodesics. Then there exists $\epsilon > 0$ such that for $|s|<\epsilon$ the closed geodesics $\alpha(s,t)$ are openly 1/k-geodesics.
\end{thm}

In Section~\ref{nsmooth} we consider the notion of a one-sided directional derivative $D^+ \colon TM \times TM \to \mathbb{R}$ for the Riemannian distance function $d \colon M \times M \to  \mathbb{R}$ (see Definition~\ref{d2_definition} and Lemma~\ref{direction_deriv2}). We apply this derivative to study the minimizing properties of curves which are close in a variational sense to an arbitrary 1/k-geodesic. In order to compare the minimizing properties of various closed geodesics we use the following notion.

\begin{definition}[\cite{Sor}, Definition 3.3] The \emph{minimizing index} of a closed geodesic $\gamma \colon S^1 \to M$, denoted $\emph{minind}(\gamma)$, is the smallest integer $k \geq 2$ such that $\gamma$ is a 1/k-geodesic.
\end{definition}

In comparing the minimizing properties of closed geodesics, we note that a lower minimizing index means that the closed geodesic minimizes on larger subintervals of length, hence has better minimizing properties. A closed geodesic $\gamma \colon S^1 \to M$ achieves the best possible minimizing property, $\text{minind}(\gamma)=2$, precisely when it minimizes between every pair $(\gamma(t),\gamma(t+\pi))$. 


\begin{thm}\label{strict_var} Let $\gamma \colon S^1 \to M$ have $\emph{minind}(\gamma)=k$. Let $\alpha(s,t)$ be a variation of $\gamma$ through closed geodesics with associated Jacobi field $J(t)$ satisfying $\langle J(t) , \dot{\gamma}(t) \rangle =0$. If there exist $p=\gamma(t_0)$ and $q=\gamma(t_0+2\pi/k)$ with 
\begin{equation}\label{eq1}
D^+_{(J(t_0),J(t_0+2\pi/k))} d \, (p,q)<0
\end{equation}
then the pair $(p,q)$ are cut points (i.e.~$\gamma$ is not an openly 1/k-geodesic) and there exists an $\epsilon >0$ such that for $0<s<\epsilon$ the closed geodesics $\alpha(s,t)$ in the variation have $\emph{minind}(\alpha(s,t)) =k+1$.
\end{thm}

In \cite{Ade2} the first author addressed the non-differentiability of the uniform energy at those $\bar{x} \in M^k$ which contain pairs $(x_i,x_{i+1})$ of cut points by introducing a generalized notion of critical point for the uniform energy. In the current paper we introduce a one-sided directional derivative for the uniform energy (see Lemma~\ref{direction_deriv3}) that allows us to identify gradient-like vectors at these points of non-differentiability (Definition~\ref{grad_like}). These directions can be used to restart the negative gradient flow of the Morse energy at its critical points, a process we illustrate on a standard flat $n$-torus.



\begin{thm}\label{2torus_thm} Let $\mathbb{T}^n$ be a standard flat $n$-torus and $\gamma \colon S^1 \to \mathbb{T}^n$  a closed geodesic with $\emph{minind}(\gamma)=k$. Using the negative of a gradient-like vector of the uniform energy $E^{k} \colon (\mathbb{T}^n)^{k} \to \mathbb{R}$ to restart the negative gradient flow of the Morse energy yields a closed geodesic $\sigma \colon S^1 \to \mathbb{T}^n$ with $\emph{minind}(\sigma) \leq  k$.
\end{thm}

We note that the ideas in this paper extend naturally to the Alexandrov setting and are therefore related to results on more general metric spaces. In particular, a notion of conjugate point for length spaces is introduced in \cite{Shan} and we anticipate a relationship with the one-sided directional derivatives examined in this paper. Additionally, our restarted negative gradient flow of the Morse energy is related to the notion of discrete homotopy introduced in \cite{Bere}, see also \cite[Example 44]{Wil}.

\section{Smooth Critical Points}\label{sec_smooth}



In this section we address Sormani's question \cite[Remark 10.5]{Sor} concerning the relationship between the degenerate smooth critical points of the uniform energy and the minimizing properties of those curves which are close in a variational sense to an openly 1/k-geodesic. We demonstrate such a relationship in Remark~\ref{rem1} by first proving Theorem~\ref{open_var}.

\begin{definition}\label{scp} A \emph{smooth critical point} of the uniform energy is a point $\bar{x} \in M^k$ at which the function $E^k$ is smooth and its gradient is zero,  i.e.~$\nabla E^k(\bar{x})=\bar{0}$.
\end{definition}

We relate the smooth critical points to closed geodesics, thereby relating the domains of the uniform energy and the Morse energy function.

\begin{lemma}[\cite{Sor}, Lemma 10.2]\label{assoc} Let $\bar{x} = (x_1, \ldots , x_k)  \in M^k $ be a smooth critical point of the uniform energy. Then there exists a unique associated closed geodesic $\gamma \colon S^1 \to M$ with $\gamma(2\pi i /k)=x_i$. 
\end{lemma}

The uniform energy is defined as a sum of squared distance functions, hence its differentiability will depend on the differentiability of its component distance functions. At points $\bar{x} \in M^k$ which do not contain pairs $(x_i, x_{i+1})$ of cut points the component distance functions are smooth (cf.~\cite{Sak}, Proposition 4.8) and it follows that the uniform energy is a smooth function ($C^{\infty}$-differentiable). A standard technique from Morse theory is to use the injectivity radius to restrict the uniform energy to a domain on which it is smooth. The uniform energy is smoothly equivalent to this restricted function in a neighborhood of a smooth critical point and the differentiability of the two functions therefore agree at such points.  The following theorem reformulates the standard relationship between the Hessians of the Morse energy function and its restricted function. Recall that the Hessian of the uniform energy at a smooth critical point $\bar{x}\in M^k$ is a symmetric bilinear form on the tangent space $T_{\bar{x}}M^k$ given by $\emph{Hess} (E^k)(\bar{x})(\bar{v},\bar{w})=\langle \nabla_{\bar{v}} \emph{grad}(E^k),\bar{w} \rangle.$




\begin{thm}[\cite{Mil}, Theorem 16.2]\label{ue_diff} The index and nullity of the Hessian of the uniform energy at a smooth critical point equal the index and nullity of the Hessian of the Morse energy at the unique associated closed geodesic. 
\end{thm}

The nullity of the Hessian of the Morse energy is equal to the number of linearly independent periodic Jacobi fields along the curve (cf.~\cite{BTZ, GM}). The Hessian of the Morse energy is always degenerate with respect to the vector field $\dot{\gamma}(t)$ and we note likewise that the Hessian of the uniform energy at a smooth critical point $\bar{x}$ is always degenerate with respect to the vector $\bar{\gamma}= [\dot{\gamma}(t_1), \ldots, \dot{\gamma}(t_k) ] \in T_{\bar{x}}M^k$. In order to discuss variations through closed geodesics we restrict our attention to the perpendicular periodic Jacobi fields, and therefore exclude the rotationally degenerate direction in the following definition. 

\begin{definition}\label{degenerate} A smooth critical point $\bar{x} \in M^k$ is said to be \emph{degenerate} if there exists $\bar{v} \in T_{\bar{x}}M^k$ such that $\nabla_{\bar{v}} \emph{grad}(E^k)(\bar{x}) =\bar{0}$ and $\langle v_i, \dot{\gamma}(t_i) \rangle =0$ for every $i = 1, \ldots , k$.
\end{definition}





%
%
%

\begin{proof}[Proof of Theorem~\ref{open_var}] 
Assume by contradiction that no such $\epsilon >0$ exists. Let $L$ be the length of the curves in the variation. Then there exists a sequence $p_n = \alpha(s_n,t_n)$ with $s_n \to 0$ such that
\begin{align*}
f(p_n, v_n) \leq L / k, \qquad v_n = \frac{\frac{\partial \alpha}{\partial t}(s_n, t_n)}{|\frac{\partial \alpha}{\partial t}(s_n, t_n)|}
\end{align*}
where $f \colon SM \to \mathbb{R} \cup \infty$ is the distance to the cut point along a geodesic with initial conditions $(p, v) \in SM$ (cf.~\cite{doC}, Proposition 13.2.9). Since this function is continuous and $S^1$ is compact, passing to a subsequence if necessary, we have $t_n \to t_0$. Setting $$p_0=\alpha(0, t_0) \qquad v_0 =  \frac{\frac{\partial \alpha}{\partial t}(0, t_0)}{|\frac{\partial \alpha}{\partial t}(0, t_0)|}$$ we have by continuity that $$v_0 = \lim_{n \to \infty} v_n = \dot{\gamma}(t_0)/|\dot{\gamma}(t_0)|$$
The fact that $\gamma$ is an openly 1/k-geodesic implies that $f(\gamma(t),  \dot{\gamma}(t)/|\dot{\gamma}(t)| )> L/k $ for every $t \in S^1$. By the continuity of $f$ we have
\begin{align*}
f(p_0, v_0) = \lim_{n \to \infty} f(p_n, v_n) \leq L / k
\end{align*}
and we have reached a contradiction. 
\end{proof}

\begin{remark}\normalfont \label{rem1} We now use Theorem~\ref{open_var} to address Sormani's open question regarding the relationship between the degenerate smooth critical points of the uniform energy and the minimizing properties of those curves which are close in a variational sense to an openly 1/k-geodesic \cite[Remark 10.5]{Sor}. Let $\bar{x} \in M^k$ be a degenerate smooth critical point of the uniform energy. By Theorem~\ref{ue_diff} we know that the unique closed geodesic $\gamma$ associated to $\bar{x}$ admits a perpendicular periodic Jacobi field $J$. If $\gamma$ is an openly 1/k-geodesic, and if the variation of $\gamma$ with variational field $J$ is through closed geodesics, then by Theorem~\ref{open_var} we conclude that this variation is through openly 1/k-geodesics, i.e.~that the minimizing properties of $\gamma$ are preserved through the variation. We note that not every perpendicular periodic Jacobi field yields a variation through closed geodesics, and that this assumption is necessary in order to discuss the minimizing properties of these nearby curves in the variation. 
\end{remark}

%
%
%
\section{Directional Derivative of the Distance Function}\label{nsmooth}

In this section we explore the notion of a one-sided directional derivative for the Riemannian distance function $d \colon M \times M \to  \mathbb{R}$. This directional derivative is then applied to study the minimizing properties of those closed geodesics which are close in a variational sense to an arbitrary 1/k-geodesic (see Theorem~\ref{strict_var}). The directional derivative of the distance function will additionally be applied in Section~\ref{ue_flow} to define a one-sided directional derivative for the uniform energy. 

\begin{definition} \label{d2_definition} The \emph{one-sided directional derivative} of the distance function $d \colon M \times M \to  \mathbb{R}$ at the point $(p,q) \in M \times M$ in the direction  $(v,w) \in T_{(p,q)}M \times M$ is defined to be
\begin{equation*}\label{d2_def} D^+_{(v,w)} d \, (p,q) = \lim_{t \to 0^+} \frac{d(  \exp_p tv , \exp_q t w  ) - d(p,q)}{t}
\end{equation*}
provided that this limit exists. 
\end{definition}

The next lemma follows from the first variation formula and states that the above limit always exists. For an exposition of these ideas, albeit in the nonpositive and nonnegative length space setting, see \cite[Section 4.5]{BBI}.

\begin{lemma}\label{direction_deriv2} Let $A(p,q)$ denote the set of unit-speed minimizing geodesics from $p$ to $q$ and set $d(p,q)=l$. Then for a vector $(v,w) \in T_{(p,q)}M \times M$ the one-sided derivative of $d \colon M \times M \to  \mathbb{R}$ in the direction $(v,w)$ exists and is given by
\begin{equation*}\label{direc_deriv2} D^+_{(v,w)} d \, (p,q) = \min{ \{ - \langle v,\dot{\eta}(0)  \rangle_p -  \langle w, -\dot{\eta}(l)  \rangle_q    :  \eta \in A(p,q)   \}. } 
\end{equation*}
\end{lemma}

\begin{proof}
Since compact Riemannian manifolds are Alexandrov spaces of curvature bounded below by some $k$, it suffices to prove the lemma for such spaces. The following indicates how the arguments presented in \cite[Section 4.5]{BBI} may be modified to provide a proof. Let $X$ be such a space and fix $p,q \in X$. Let $\gamma_1 \colon [0,T] \to X$ and $\gamma_2 \colon [0,T] \to X$ be two unit speed geodesics with $\gamma_1(0) = p$ and $\gamma_2(0)=q $. Suppose that for each $t \in [0,T]$, $\sigma_t \colon [0,1] \to X$ is a shortest path connecting $\gamma_1(t)$ with $\gamma_2(t)$, and $\{ \sigma_{t_i} \}$ is a sequence that converges to $\sigma_0$. Set $\ell(t) = d(\gamma_1(t),\gamma_2(t))$. Then we claim that
\begin{align} \label{eq:first variation formula}
\lim_{i \to \infty} \frac{\ell(t_i) - \ell(0)}{t_i} = -\cos \alpha_1 - \cos \alpha_2
\end{align}
where $\alpha_i$ is the angle made by $\sigma_0$ with $\gamma_i$. We note that \cite[Theorem 4.5.6]{BBI} is the special case of \eqref{eq:first variation formula} when $k = 0$ and $\gamma_2$ is the constant path at $q$. Since equation (4.3) in the proof of \cite[Theorem 4.5.6]{BBI} holds for spaces with an arbitrary curvature bound, a straightforward modification of the inequalities shows that \eqref{eq:first variation formula} holds for arbitrary lower curvature bound and two variable endpoints.

Now let $B(p,q)$ denote the set of minimizing paths from $p$ to $q$ and define
\begin{align*}
B := \min \{ - \cos(\ma(\gamma_1,\sigma_0)) - \cos(\ma(\gamma_2,\sigma_0)) \ | \ \sigma_0 \in B(p,q) \}.
\end{align*}
With almost no modification, the proof of \cite[Corollary 4.5.7]{BBI} shows that
\begin{align*} \label{eq:one-sided derivative formula}
& \lim_{t \to 0^+} \frac{\ell(t) - \ell(0)}{t} = B.
\end{align*}
When $X$ is a compact Riemannian manifold, the $B$ defined here can be expressed as the minimum from the statement of the lemma, and the result follows.
\end{proof}

\begin{figure}[!htb]\centering
    \includegraphics[width=9.5cm, height=3.6cm]{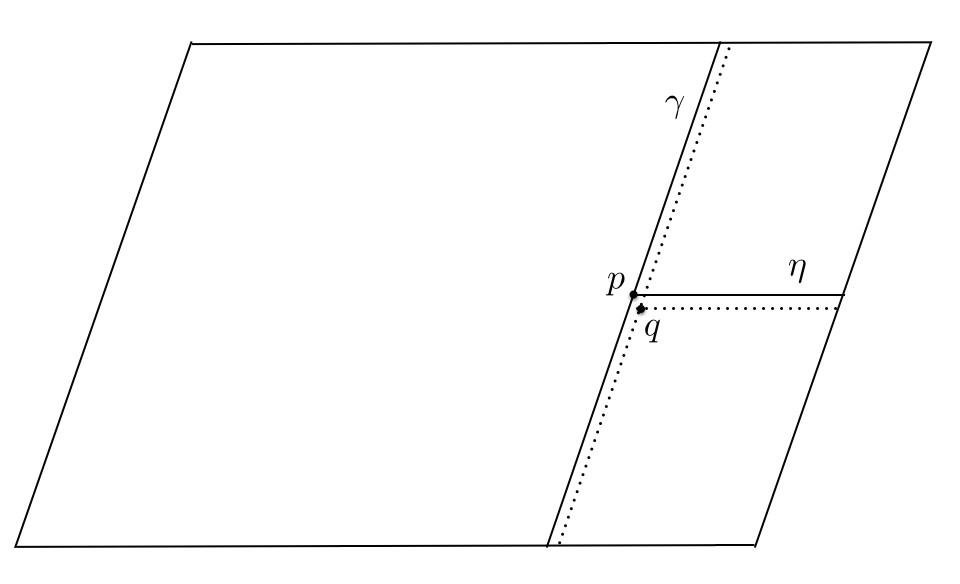}
    \caption{The doubled rectangle from Example~\ref{double_ex}. Note that $p$ is on the top face of the doubled rectangle, whereas $q$ is on the bottom face.} 
    \label{double}
    \end{figure}

\begin{ex}\normalfont \label{double_ex} The doubled rectangle $M$ in Figure~\ref{double} illustrates Theorem~\ref{strict_var}. The over-under geodesic $\gamma$ has $\text{minind}(\gamma)=2$. The pair of points $(p,q)$ on $\gamma$ are connected via another minimizing geodesic, $\eta \colon [0,\pi] \to M$, which connects the point $p = \eta(0)=\gamma(t_0)$ on the top face to the point $q = \eta(\pi)=\gamma(t_0+\pi)$ on the bottom face. The closed geodesic $\gamma$ admits a variation $\alpha(s,t)$ through closed geodesics simply by sliding $\gamma$ left or right along the doubled rectangle. The Jacobi field associated to this variation coincides with the velocity vector of $\eta$, i.e.~$J(t_0)=\dot{\eta}(0)$ and $J(t_0+\pi)= -\dot{\eta}(\pi)$. Computing via Lemma~\ref{direction_deriv2} we have $$D^+_{(J(t_0),J(t_0+\pi))} d \, (p,q) = - \langle J(t_0),\dot{\eta}(0)  \rangle_p -  \langle J(t_0+\pi), -\dot{\eta}(\pi)  \rangle_q =-2$$ so that by Theorem~\ref{strict_var} we can conclude that $\text{minind}(\alpha(s,t)) =3$ for $0<s<\epsilon$. Indeed, these curves $\alpha(s,t)$ fail to minimize between the pairs $(\alpha(s,t_0),\alpha(s,t_0+\pi))$ as $\eta$ now provides a shorter path between these points. Also worthy of note, although not a consequence of the theorem, is that $\text{minind}(\alpha(s,t)) =2$ for $-\epsilon<s<0$, i.e.~sliding $\gamma$ to the left does not change its minimizing index. 
\end{ex}

\begin{proof}[Proof of Theorem~\ref{strict_var}] We first note that Equation \eqref{eq1} and Lemma~\ref{direction_deriv2} together with the fact that $\langle J(t) , \dot{\gamma}(t) \rangle =0$ imply that there are multiple minimizing geodesics joining $p$ and $q$, and therefore that the pair are cut points.

Next let $L$ be the length of the curves in the variation. By Equation \eqref{eq1} there exists an $\epsilon >0$ such that for $0<s<\epsilon$ we have $$d(\alpha(s,t_0),\alpha(s,t_0+2\pi/k)) < L/k.$$ This condition implies that these $\alpha(s,t)$ are not 1/k-geodesics. By continuity of the function $f \colon SM \to \mathbb{R} \cup \infty$ (see the proof of Theorem~\ref{open_var}) we have $\text{minind}(\alpha(s,t)) =k+1$.
\end{proof}

%
%
%
\section{Gradient Flow for the Uniform Energy}\label{ue_flow}

The negative gradient flow of the Morse energy function has been used to demonstrate the existence of closed geodesics on non-simply connected manifolds by deforming closed curves within their homotopy class in the direction of decreasing energy. The flow terminates at the critical points of the Morse energy function, the closed geodesics. In this section we develop a procedure by which the negative gradient flow can be restarted from these critical points. We prove Theorem~\ref{2torus_thm} which demonstrates this procedure on a standard flat $n$-torus. In this setting, we see that the extended negative gradient flow need not respect homotopy class and works to improve the minimizing index of the associated closed geodesics. 

On non-simply connected manifolds the shortest non-contractible closed curve has minimizing index 2 \cite[Lemma 4.1]{Sor}. In the simply connected setting, metrics on the 2-sphere have been constructed for every $k \geq 2$ that fail to admit $1/k$-geodesics (see \cite{Ho} and also \cite{Ade}). As Theorem~\ref{2torus_thm} illustrates, the procedure developed in this section works in some cases to improve the minimizing index of the associated closed geodesics, and therefore could be used to identify the minimal minimizing index $k$ on a given simply connected manifold. This value $k$ can then be used to provide an upper bound of $k \cdot diam(M)$ on the length of the shortest closed geodesic (cf. \cite[Lemma 3.2]{Sor}). 

We now briefly sketch this restarting procedure, with additional details given throughout the section. Start by identifying a pair of cut points $(x_1, x_2)$ on the given closed geodesic which are at maximal distance for such pairs. Using the smallest integer $k \geq 2$ such that the given closed geodesic is a 1/k-geodesic, identify the point $\bar{x} = (x_1, \ldots , x_k) \in M^k$ where the points $ \{ x_3, \ldots , x_k \}$ are chosen to be evenly spaced along the remainder of the geodesic. Note by construction that $d(x_1,x_2) \geq L/k \geq d(x_i, x_{i+1})$ for $i \neq 1$, where $L$ is the length of the closed geodesic. Next, perturb $\bar{x}$ in the negative of a gradient-like direction (see Definition~\ref{grad_like}) and reconnect the perturbed pairs $(x_i, x_{i+1})$ via minimal geodesics. The resulting curve will be a piecewise smooth closed curve from which one can restart the negative gradient flow of the Morse energy function.

At smooth critical points $\bar{x} \in M^k$ of the uniform energy we have that $\nabla E^k(\bar{x})=\bar{0}$ and are unable to identify directions in which to restart the negative gradient flow. We therefore focus on those $\bar{x} \in M^k$ which contain pairs $(x_i,x_{i+1})$ of cut points. At such $\bar{x} \in M^k$ the Hessian of the uniform energy is not defined and the second derivative test is not available.  We introduce a one-sided directional derivative for the uniform energy that will allow us to identify gradient-like vectors at these points, and hence directions of maximal decrease for the uniform energy. 




\begin{definition} \label{de_definition} The \emph{one-sided directional derivative} of the uniform energy $E^k \colon M^k \to \mathbb{R}$ at the point $\bar{x} \in M^k$ in the direction $\bar{v} \in T_{\bar{x}} M^k$ is defined to be
\begin{equation*}\label{de_def} D^+_{\bar{v}} E^k (\bar{x}) = \lim_{t \to 0^+} \frac{  E^k(  \exp_{\bar{x}} t \bar{v} )  - E^k(\bar{x})  }{t}
\end{equation*}
provided that this limit exists. 
\end{definition}

\begin{lemma}\label{direction_deriv3} Let $\bar{x} \in M^k$ with $d(x_i, x_{i+1})=l_i$. For a vector $\bar{v} \in T_{\bar{x}} M^k$ the one-sided derivative of $E^k \colon M^k \to \mathbb{R}$ in the direction $\bar{v}$ exists and is given by 
\begin{align*}\label{direc_deriv3} & \tfrac{1}{2k} D^+_{\bar{v}} E^k \, (\bar{x}) =\\ 
&  \min   { \{ - \langle \, \bar{v} , ( \ldots, - l_{i-1} \cdot \dot{\eta}_{i-1}(l_{i-1}) + l_i \cdot \dot{\eta}_{i}(0) , - l_i \cdot  \dot{\eta}_{i}(l_i) + l_{i+1} \cdot \dot{\eta}_{i+1}(0) , \ldots) \, \rangle    \} }
\end{align*}
where the \emph{min} is taken over all k-tuples $(\eta_1, \ldots, \eta_k)$ such that $\eta_i \colon [0,l_i] \to M$ are unit-speed minimizing geodesics with $\eta_i(0)=x_i$ and $\eta_i(l_i)=x_{i+1}$.
\end{lemma}
\begin{proof} This follows directly from Lemma~\ref{direction_deriv2} and the definition of the uniform energy as a sum of squared distance functions.
\end{proof}

\begin{definition}\label{cand_grad} The \emph{candidate gradients} for the uniform energy at $\bar{x} \in M^k$ are the vectors $$-( \ldots, - l_{i-1} \cdot \dot{\eta}_{i-1}(l_{i-1}) + l_i \cdot \dot{\eta}_{i}(0) , - l_i \cdot  \dot{\eta}_{i}(l_i) + l_{i+1} \cdot \dot{\eta}_{i+1}(0) , \ldots) \in T_{\bar{x}} M^k$$ where the $\eta_i \colon [0,l_i] \to M$ are the unit-speed minimizing geodesics from above. 
\end{definition}

The following proposition, which follows directly from the definition, serves to illuminate the geometric significance of the candidate gradients. We say that a closed geodesic $\gamma \colon S^1 \to M$ is \emph{associated} to the point $\bar{x} \in M^k$ if there exist $t_i \in S^1$ with $\gamma(t_i) = x_i$ for every $i \in \{1, \ldots, k \}$. 

\begin{prop}\label{assoc_cand} A point $\bar{x} \in M^k$ has an associated closed geodesic if and only if $\bar{0} \in T_{\bar{x}} M^k$ is a candidate gradient at $\bar{x}$.
\end{prop}

We are interested in determining the directions $\bar{v} \in T_{\bar{x}} M^k$ which maximize the directional derivative of the uniform energy. We therefore choose the candidate gradient with maximal magnitude in the definition below.

\begin{definition}\label{grad_like} A \emph{gradient-like} vector $\bar{v} \in T_{\bar{x}} M^k$ for the uniform energy at a point $\bar{x}$ is a candidate gradient which has maximal magnitude, where the magnitude is calculated using the Riemannian product structure on $M^k$.
\end{definition}

Note that $\bar{0} \in T_{\bar{x}} M^k$ will be the gradient-like vector precisely when it is the only candidate gradient. In particular, this occurs when the point $\bar{x}$ is a smooth critical point. 

This paper concludes with the proof of Theorem~\ref{2torus_thm} which illustrates the process of using a gradient-like vector on a standard flat $n$-torus to restart the negative gradient flow of the Morse energy and improve the minimizing properties of the associated closed geodesics. We first continue with the necessary preliminaries. 

\begin{lemma}\label{product_lemma} Let $M = M_1 \times \cdots \times M_n$ be a Riemannian product with $\gamma = (\gamma_1 , \ldots,  \gamma_n)$ a geodesic. Then $\gamma$ minimizes between $p=(p_1, \ldots, p_n)=\gamma(0)$ and $q=(q_1, \ldots ,q_n)=\gamma(t)$ if and only if $\gamma_i$ minimizes between $p_i=\gamma_i(0)$ and $q_i=\gamma_i(t)$ for every $i = 1, \ldots ,n $. 
\end{lemma}

\begin{prop}\label{product} Let $M = M_1 \times \cdots \times M_n$ be a Riemannian product with $\gamma = (\gamma_1 , \ldots,  \gamma_n) \colon S^1 \to M$ a closed geodesic. Then $$\emph{minind}_M(\gamma)=\max{ \{ \emph{minind}_{M_i}(\gamma_i) \}}.$$
\end{prop}
\begin{proof}
Without loss of generality assume $ \text{minind}_{M_1}(\gamma_1) =\max{ \{ \text{minind}_{M_i}(\gamma_i) \}}$ and set $k = \text{minind}_{M_1}(\gamma_1) $. Then $\gamma_i$ minimizes between $\gamma_i(t)$ and $\gamma_i(t+2\pi/k)$ for all $t\in S^1$ and $i = 1, \ldots, n$. By the lemma $\gamma$ minimizes between $\gamma(t)$ and $\gamma(t+2\pi/k)$ for all $t\in S^1$ and we have $\text{minind}_{M}(\gamma) \leq k$. 

Assume by contradiction that $\text{minind}_{M}(\gamma) < k$. Then $\gamma$ minimizes between $\gamma(t)$ and $\gamma(t+2\pi/(k-1))$ for all $t\in S^1$. By the lemma $\gamma_1$ minimizes between $\gamma_1(t)$ and $\gamma_1(t+2\pi/(k-1))$ for all $t\in S^1$ so that $ \text{minind}_{M_1}(\gamma_1) < k$ a contradiction.  
\end{proof}

\begin{definition} Let $\mathbb{T}^n$ be a standard flat $n$-torus and $\gamma \colon S^1 \to \mathbb{T}^n$ a closed geodesic representing the free homotopy class $[(m_1, \ldots, m_n)] \in \pi_1(\mathbb{T}^n)$. We will call such a curve an \emph{$(m_1, \ldots ,m_n)$-geodesic}.
\end{definition}

\begin{ex}\normalfont \label{2torus_ex} Let $\mathbb{T}^2$ be a flat rectangular two-torus and $\gamma \colon S^1 \to \mathbb{T}^2$ the $(1,2)$-geodesic with associated point $\bar{x} \in (\mathbb{T}^2)^4$ as shown in Figure~\ref{2torus}. As a closed geodesic $\gamma$ is a critical point of the Morse energy function and is therefore a fixed point for the negative gradient flow. We will use a gradient-like vector for the uniform energy at $\bar{x}$ to restart the negative gradient flow. 

\begin{figure}[!htb]\centering
    \includegraphics[width=.6 \textwidth]{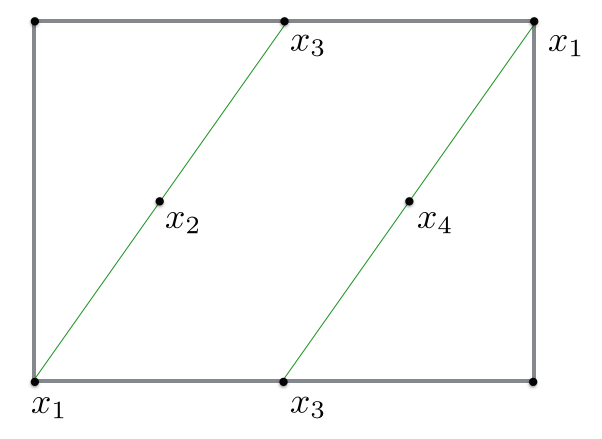}
    \caption{The flat rectangular two-torus $\mathbb{T}^2$. The green line is the $(1,2)$-geodesic with associated point $\bar{x}=(x_1,x_2,x_3,x_4) \in (\mathbb{T}^2)^4$. } 
    \label{2torus}
    \end{figure}

We first determine the candidate gradients and the gradient-like vectors for the point $\bar{x} \in (\mathbb{T}^2)^4$. There are two minimizing geodesics between each pair $(x_i, x_{i+1})$. The two initial velocity vectors of these geodesics at $x_i$ sum with the two initial velocity vectors of the minimizing geodesics between $(x_{i-1},x_i)$ to yield three potential gradient directions at $x_i$: $\{ \pm j,0 \} \in T_{x_i}\mathbb{T}^2$, where $j$ is a basis vector in the vertical direction. These directions combine via Lemma~\ref{direction_deriv3} to yield the following candidate gradients at $\bar{x}$:
\begin{align*} &\pm \langle  j,  -j, 0, 0 \rangle , \pm \langle  0, j,  -j, 0 \rangle , \pm \langle  0,0,j,  -j \rangle , \pm \langle  -j, 0,0, j \rangle 
\\ &\pm \langle j, -j, j, -j \rangle , \pm \langle j, 0, -j, 0 \rangle , \pm \langle 0, j, 0, -j \rangle , \langle 0,0,0,0 \rangle
\end{align*}
The gradient-like vectors are therefore $\pm \langle j, -j, j, -j \rangle$. We will use the negative of these directions, or the directions of maximal decrease of the uniform energy, to restart the negative gradient flow of the Morse energy function in an attempt to improve the minimizing properties of the closed geodesic. Figure~\ref{2torus2} illustrates this process when using the direction $\langle j, -j, j, -j \rangle$. After restarting the flow it stabilizes at a new critical point, the $(1,0)$-geodesic $\sigma \colon S^1 \to \mathbb{T}^2$.  We have $$\text{minind}(\sigma)=2 < 4 = \text{minind}(\gamma)$$ so that this restarted negative gradient flow has improved the minimizing properties of the associated closed geodesics.

\begin{figure}[!htb]\centering
    \includegraphics[width=.62 \textwidth]{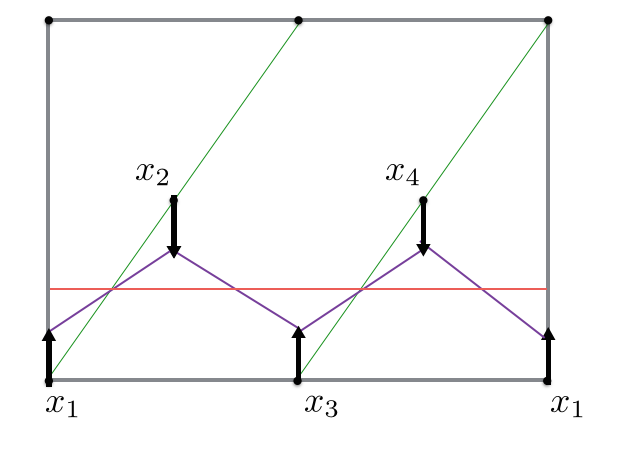}
    \caption{The torus $\mathbb{T}^2$ with the negative gradient-like vector $\langle j, -j, j, -j \rangle \in T_{\bar{x}}(\mathbb{T}^2)^4$. The purple curve is a dramatization of the piecewise smooth closed curve at which we would restart the negative gradient flow, and the red curve depicts the new critical point, a $(1,0)$-geodesic. } 
    \label{2torus2}
    \end{figure}

\end{ex}

\begin{proof}[Proof of Theorem~\ref{2torus_thm}] Let $(y_1, \ldots, y_n)$ be standard coordinates on $\mathbb{T}^n$ with associated basis of tangent vectors $(\partial y_1, \ldots , \partial y_n)$. Let $\gamma$ be a $(m_1, \ldots, m_n)$-geodesic and without loss of generality assume that $m_1= \max{ \{m_i \} }$ so that $k=\text{minind}(\gamma)=2\,m_1$. Then any choice $\bar{x} \in (\mathbb{T}^n)^{k}$ of evenly spaced points on $\gamma$ has the property that every pair $(x_i, x_{i+1})$ is a pair of cut points. Moreover we have that $(\partial y_1, -\partial y_1, \ldots,\partial y_1, -\partial y_1 ) \in T_{\bar{x}} (\mathbb{T}^n)^{k}$ is a gradient-like vector. In order to restart the negative gradient flow of the Morse energy from this critical point, we first perturb $\bar{x}$ in the negative of this gradient-like direction and reconnect the perturbed $(x_i, x_{i+1})$ via minimal geodesics. The resulting curve is a piecewise smooth geodesic which is homotopically trivial in the $y_1$ direction. Restarting the negative gradient flow of the Morse energy from this piecewise smooth geodesic results in a $(0, m_2, \ldots, m_n)$-geodesic with $\text{minind} \leq k$ . Iterating this process yields a closed geodesic $\sigma \colon S^1 \to \mathbb{T}^n$ with $\text{minind}(\sigma) = \min{\{m_i \}} \leq k$. 
\end{proof}

%
%
%
\section{Acknowledgements} 
The authors would like to thank Carolyn Gordon and Craig Sutton for their guidance through the research process. The authors would also like to thank Christina Sormani for suggesting the original problem and for helpful discussions and direction. Finally, the authors would like to thank the referee for many insightful comments. 


\bibliographystyle{abbrv}
\bibliography{epstein}

\end{document}